\documentclass[letterpaper, 10 pt, conference]{ieeeconf}

\usepackage{amsmath,amssymb}
\usepackage{graphicx}
\usepackage[dvipsnames]{xcolor}
\usepackage{enumerate}

\IEEEoverridecommandlockouts 

\newcommand{\real}{\ensuremath{\mathbb{R}}}
\newcommand{\C}{\ensuremath{\mathcal{C}}}
\newcommand{\D}{\ensuremath{\mathcal{D}}}
\newcommand{\E}{\ensuremath{\mathcal{E}}}
\newcommand{\G}{\ensuremath{\mathcal{G}}}
\newcommand{\V}{\ensuremath{\mathcal{V}}}

\newcommand{\neigh}{\ensuremath{\mathcal N_i}}
\DeclareMathOperator{\col}{\mathrm{col}}
\DeclareMathOperator{\diag}{\mathrm{diag}}

\newtheorem{theorem}{Theorem}
\newtheorem{property}{Property}

\newtheorem{assumption}{Assumption}
\newtheorem{lemma}{Lemma}

\newtheorem{remark}{Remark}
\newtheorem{problem}{Problem}
\newtheorem{proposition}{Proposition}

\newcommand\smallmat[1]{\left[\begin{smallmatrix}#1\end{smallmatrix}\right]}
\newcommand\infnorm[1]{\|#1\|_\infty}

\newcounter{deff}
\newenvironment{deff}[1]{\refstepcounter{deff}\indent \textit{Definition \thedeff \ (#1):} }{\hfill}

\title{\LARGE \bf
Distributed hybrid observer with prescribed convergence rate for a linear plant using multi-hop decomposition}

\author{Riccardo Bertollo$^1$, Pablo Mill\'{a}n$^2$, Luis Orihuela$^2$, Alexandre
Seuret$^3$, Luca Zaccarian$^4$%
\thanks{$^1$R. Bertollo is with the 
University of Trento, Italy (e-mail: riccardo.bertollo@unitn.it).
$^2$P. Mill\'{a}n and L. Orihuela are with 
Universidad Loyola Andaluc\'{i}a, Spain (e-mail: pmillan,dorihuela@uloyola.es).
$^3$A. Seuret is with Universidad de Sevilla, 
Spain (e-mail: aseuret@us.es), 
$^4$L. Zaccarian is with LAAS-CNRS,  University of Toulouse, CNRS, Toulouse, France and the University of Trento, Italy (e-mail: lzaccarian@laas.fr).}%
\thanks{Research supported in part by ANR via grant HANDY (number ANR-18-CE40-0010), by project METRICA (Grant PID2020-117800GB-I00 funded by AEI/MCeI) and project IRRIGATE (Grant PY20 RE 017 funded by Junta de Andalucía).}}

\begin{document}

  \maketitle
  \thispagestyle{empty}
  \pagestyle{empty}

  \begin{abstract}
    We propose a distributed hybrid observer for a sensor network where the plant and local observers run in continuous time and the information exchange among the sensing nodes is sampled-data. Process disturbances, measurement noise and communication noise are considered, and we prove that under some necessary detectability assumptions the observer gains can be tuned to guarantee exponential Input-to-State Stability with a prescribed convergence rate. Simulations illustrate the performance of the proposed observer.
  \end{abstract}
  
  \begin{keywords}
  Distributed estimation; Hybrid observer; Linear systems; Input-to-State Stability; Sensor network
  \end{keywords}

  \section{Introduction}
    
  
  One of the most important tasks for a Wireless Sensor Network (WSN) is state estimation, which finds important applications in different applications areas such as robotics, power systems, smart grids or traffic management, among others.
  In distributed state estimation, 
  each smart sensor (referred usually as \emph{agent}) estimates the plant state with limited locally available information, possibly coming from seldom information exchange with neighboring agents.
%
 %
  Perhaps the most popular approach for this setting is the Distributed Kalman Filter (DKF) \cite{olfati2007}. DKF algorithms involve several steps that typically include: i) prediction of prior states using the system model, ii) correction based on local measurements with a Kalman gain, and iii) information fusion from neighboring agents based on consensus \cite{carli2007,battistelli2014,he2019distributed}, diffusion \cite{cattivelli2010}, or gossip-based algorithms \cite{kar2010}. An excellent review of these techniques can be found in \cite{he2020}. 
    
  Unlike DKF, distributed Luenberger observers do not use statistical noise/disturbance descriptions \cite{park2016design} and resort to deterministic approaches. To solve the problem in a distributed way, different state-based partitions have been proposed. In \cite{mitra2018distributed}, a new estimation structure based on the observability canonical form is given, which depends on locally measured outputs and an agent indexing or ordering. An alternative subspace decomposition method, called multi-hop decomposition, is proposed in \cite{delNozalAut19}, where each agent constructs its own observability staircase form. These works provide a fully distributed design for the estimator in discrete time \cite{park2016design,mitra2018distributed,delNozalAut19} or continuous time \cite{delNozalAut19}.
  However, the case where the local measurements are available in continuous time but the communication among the agents happens in discrete time cannot be analyzed with these approaches.
  That situation, which has been studied only in the centralized case (see e.g. \cite{etienne2017observer,Ferrante2016state,sferlazza2018time} and references therein), is certainly relevant in practice. Indeed, on the one hand, requiring continuous-time communication would be too stringent and, on the other hand, designing sampled-data solutions with a single sampling time forces the algorithm to run at the inter-agent communication rate, typically much slower than the rate at which the local outputs can be locally measured.
  
 This paper develops a new hybrid distributed observer exploiting the multi-hop decomposition \cite{delNozalAut19} for linear plants equipped with a sensor network.
 Each network node runs locally a hybrid observer in continuous time for the locally observable modes, exploiting the continuous availability of locally measured outputs through continuous Luenberger-like dynamics.
 {Our work exploits the multi-hop decomposition in \cite{delNozalAut19} for the two-time-scale setting of fast local output injection and slow sampled-data inter-agent exchange, in addition to rigorously characterizing the effect of (continuous-time) process disturbance/sensor noise and (discrete-time) communication exchange disturbances.
 Using an ad-hoc Lyapunov construction (different from \cite{delNozalAut19}), we guarantee a prescribed convergence rate and input-to-state stability (ISS) properties.}

    \emph{Notation:} Scalar $a_{ij}$ denotes the $(i,j)$ entry of matrix $A$.
Given $V \subset \mathbb{N}$ and an ordered set $Y_i$ of matrices, for decreasing values of $i$ from top to bottom, we denote by $\diag\nolimits\limits_{i \in V} \left( Y_i \right)$
a block diagonal matrix whose diagonal blocks are $Y_i$, while
for the case where matrices $Y_i$ all have the same number of columns,
 $\col\nolimits\limits_{i \in V} \left( Y_i \right)$, where $V \subset \mathbb{N}$ and $Y_i$  
 a matrix vertically stacking $Y_i$.
%
    We often denote $\left( u, v \right):= \left[ u^\top, v^\top \right]^\top$.
    $|x| := \sqrt{x^\top x}$ is the Euclidean norm of vector $x$.
    $\infnorm{x}$ denotes the $\mathcal L$-infinity norm of the signal $x(\cdot)$.
    Given two sets $\V \subset \mathbb N$ and $\E \subset \V \times \V$, $\G(\V,\E)$ is the (directed) graph with nodes $\V$ and edges $\E$.
    Finally, $\neigh$ denotes the set of nodes $j$ such that $(j,i)$ is an edge of $\G(\V,\E)$, called in-neighborhood of node $i$.

  \section{System description and problem formulation}
  \label{sec:multihop}
\subsection{System data}
Consider a linear plant subject to an external disturbance driven by the following equation
 \begin{align}\label{eq:plant}
      \begin{split}
        &\dot x = A x + d_0, \\
        &y_i = C_i x + d_{i}, \quad \forall i\in \mathcal V:= \{1,\dots,p\}
      \end{split}
    \end{align}
  where $x \in \real^n$ is the plant state, $d_0$ is some external disturbance, matrices $A$, $C_i$, for all $i\in \mathcal V$ are constant and known. 
  Each sensor measures an output $y_i$ affected by measurement noise $d_i$. As shown in Fig.~\ref{fig:block_diagram}, each output measurement $y_i \in \real^{m_i}$, with $m_i\leq n$, is available only locally at the sensor's location within our sensors network. Each pair $(C_i, A)$ is not necessarily detectable.
  \begin{figure}[t]
    \centering
    \includegraphics[width = 0.67\columnwidth]{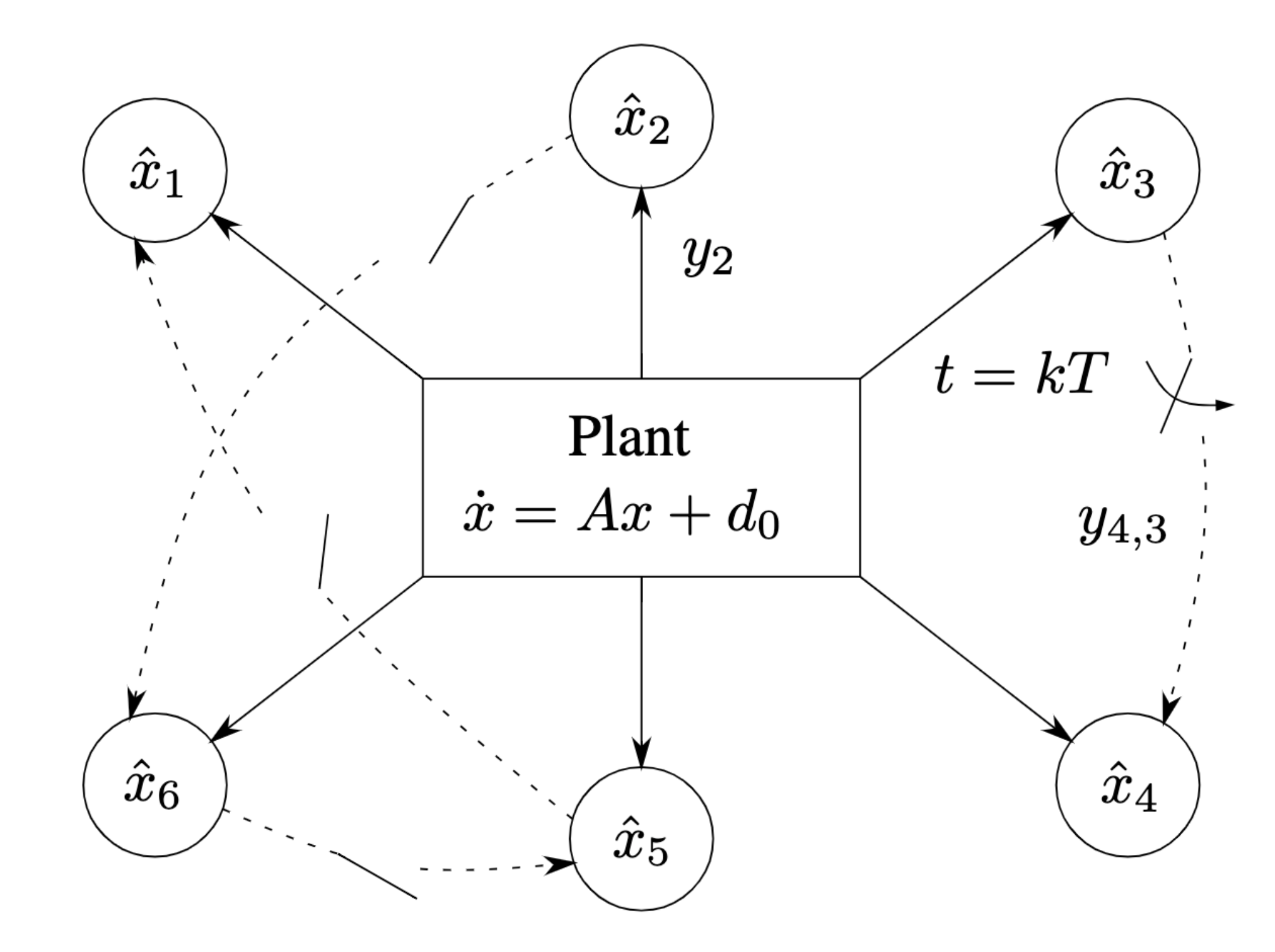}
    \caption{Sample network with $p=6$ sensors. Solid arrows are the locally available continuous-time measurements $y_i$, dashed lines are periodically sampled information
exchanges among the agents.}
    \label{fig:block_diagram}
  \end{figure}

\subsection{Cross-sensors information exchange}

In our solution each agent is equipped with a local observer providing a state estimate $\hat x_i$, $i\in\mathcal V$. Since each measurement $y_i$ is not necessarily sufficient for the local observer to reconstruct $x$, cross-sensors communication is needed. We assume that the sensors  $\V$ are connected through a communication graph $\G = \left( \V, \E \right)$, wherein 
the exchange of information is performed synchronously and periodically with a period $T>0$.
The information exchange along each edge of the graph $(i,j) \in \E$ is also affected by transmission noise representing, e.g., small delays or quantization errors.
Summarizing, each sensor $i \in {\mathcal V}$ has access to two types of measurements, namely
\begin{align}
   & y_i(t)  ,   && \forall t \in {\mathbb R} , \\
&    y_{i,j}(kT) := \hat x_j(k T) + w_{i,j}(k), &&  \forall k \in \mathbb N, \forall j \in \neigh ,
\end{align}
where the second set of measurements is only available at periodic instants of time (sampled-data) $t_k = k T$.

In the previous equation $w_{i,j}$ is a (discrete-time) communication disturbance affecting the information transmitted from sensor $j$ to sensor $i$. In the sequel, we will denote with $w_i := \col\nolimits\limits_{j \in \neigh} \left( w_{i,j} \right)$ the collection of the communication disturbances transmitted to sensor $i$, and we will collect together all $d_i$ and $w_i$ as follows
\begin{align}
\label{eq:noise}
    d := \col\nolimits\limits_{i \in \{0,\ldots,p\}} \left( d_i \right), \quad\quad
    w := \col\nolimits\limits_{i \in \V} \left( w_i \right).
\end{align}

\subsection{Problem formulation}

We design here a  pool of distributed observers only using the locally available information $y_i(t), y_{i,j}(kT)$.
Due to the disturbances acting on the plant and on the measurements, an asymptotic state estimation is not achievable. Therefore we establish an ISS bound from the disturbances \eqref{eq:noise} to the estimation errors, as stated below.
\begin{problem}
\label{problem}
Design a \textit{distributed hybrid or sampled-data observer} for system \eqref{eq:plant} such that, for a given exponential rate $\alpha > 0$, the following ISS property holds 
    \begin{align}
      \label{eq:ISS}
        \begin{split}
          \left| \col\nolimits\limits_{i \in \V} \left( \hat x_i (t) - x(t) \right) \right| &\leq \kappa \mathrm{e}^{-\alpha t} \left| \col\nolimits\limits_{i \in \V} \left( \hat x_i(0) - x(0) \right) \right| \\ 
          & \hspace{-5pt} + \gamma_C \infnorm{d} + \gamma_D \infnorm{w}, \ \forall t \in \mathbb R_{\geq 0},
        \end{split}
      \end{align}
\end{problem}
where $\hat x_i$ is the state estimate at the $i$-th node and $\kappa, \gamma_C, \gamma_D$ are some strictly positive, constant gains.

\begin{remark}
  Considering \eqref{eq:ISS} without noise we obtain a standard {Uniform Global Exponential Stability (UGES)} property for the estimation error.
  {Property \eqref{eq:ISS} implies robustness to small nonlinearities in the plant (included in $d_0$) and small delays in the communication (included in $w_{i,j}$). More general small nonlinear perturbations of plant \eqref{eq:plant} or asynchronous communications are covered by the intrinsic robustness results in \cite[Ch. 7]{TeelBook}, enabled by the well-posedness of our error dynamics.
  Quantitative robustness-in-the-large characterizations are more challenging objectives that will be the subject of future work.}
\end{remark}

\section{Continuous-discrete observer design}
\label{sec:observer_design}
  
\subsection{Multi-hop decomposition and necessary assumptions}
    The proposed distributed observer architecture relies on the multi-hop decomposition presented in \cite{delNozalAut19} for a discrete-time setting and whose main definitions are easily extended hereafter to the continuous-time case.

    For the linear system \eqref{eq:plant}, the $\varrho$-hop output matrix, $C_{i,\varrho}$, of agent $i \in \V$ is recursively defined as $C_{i,0} := C_i$ and
    \begin{align*}
      C_{i,\varrho} := \begin{bmatrix}
        C_{i,\varrho-1} \\ \col\nolimits\limits_{j \in \neigh} \left( C_{j,\varrho-1} \right)
      \end{bmatrix}.
    \end{align*}

In other words, $C_{i,\varrho}$ represents the outputs measured by agent $i$ and by all the agents $j$ such that in $\G$ there exists a (directed) path from $j$ to $i$ of length at most $\varrho$. Consider now, for any $i \in \V$ and any ``hop'' $\varrho \in \{0, \ldots, \ell_i\}$ of our multi-hop decomposition, the pair $\left( C_{i,\varrho}, A \right)$, where $\ell_i \in \mathbb N$ is specified in Assumption~\ref{ass:detectable} below.
    Based on the classical continuous-time observability concept, we can define the observable and unobservable sub-spaces of $\real^n$ related to $\left( C_{i,\varrho}, A \right)$, denoting them by $\mathcal O_{i,\varrho}$ and $\overline{\mathcal O}_{i,\varrho}$, respectively.
    Then, we can define the ``innovation'' matrix for agent $i$ at hop $\varrho \in \{0, \ldots, \ell_i\}$ as the matrix $W_{i,\varrho}$ whose columns generate the subspace of $\real^n$ that is observable at hop $\varrho$ and unobservable at hop $\varrho-1$.
    Namely, defining $\overline{\mathcal O}_{i,-1} := \mathbb R^n$, $W_{i,\varrho}$ is recursively defined as an orthogonal matrix such that
    \begin{align}
    \label{eq:W_i_rho}
        \begin{split}
          &\mathrm{Im} \left( W_{i,\varrho} \right) = \overline{\mathcal O}_{i,\varrho-1} \cap \mathcal O_{i,\varrho}, \quad\quad \varrho \in \{0, \ldots, \ell_i\} \\
          &\mathrm{Im} \left( W_{i,\ell_i+1} \right) = \overline{\mathcal O}_{\ell_i}.
        \end{split}
    \end{align}
    It is immediate to verify that horizontally stacking matrices $W_{i,\varrho}$ in \eqref{eq:W_i_rho}, $\varrho \in \{0, \ldots, \ell_i+1\}$, provides an orthogonal matrix $W_i$, whose image is $\mathbb R^n$.
    \begin{remark}
    \label{rmk:multihop}
      For linear systems, the observable sub-spaces related to $\left(C_{i,\varrho},A\right)$ are identical in the continuous-time and discrete-time cases.
      Thus, the multi-hop decomposition results in \cite[Lemma 3-10]{delNozalAut19} can and will be used hereafter. 
    \end{remark}
    
    Note that matrix $W_{i, \ell_i+1}$ in \eqref{eq:W_i_rho} might be non-empty, regardless of the value $\ell_i$.
    For this reason, some assumption on the unobservable modes is necessary for designing a distributed observer that solves Problem~1.
    In particular, given a desired exponential convergence rate $\alpha$ in \eqref{eq:ISS}, the least conservative assumption is based on the definition of collective $\alpha$-detectability presented in \cite{delNozalAut19}.
    
    \begin{deff}{collective $\alpha$-detectability \cite{delNozalAut19}}
    \label{def:multi-hop}
        Given $\alpha>0$, system \eqref{eq:plant} is collectively $\alpha$-detectable with respect to $\mathcal G$ if, for each $i \in \V$, there exists 
        $\ell_i\in\mathbb N$ such that pair $\left( C_{i,\ell_i}, A \right)$ is $\alpha-$detectable.
    \end{deff}
    \begin{assumption}
    \label{ass:detectable}
        System~\eqref{eq:plant} is collectively $\alpha$-detectable with $\alpha>0$ and $\ell_i$ in \eqref{eq:W_i_rho} is selected according to Definition~\ref{def:multi-hop}.
    \end{assumption}
    
    Assumption~\ref{ass:detectable} does not require any connectedness or spanning tree assumption on graph $\G$, which are typical assumptions in distributed estimation.
    Such assumptions may be conservative (see the example in \cite[Figure 1]{delNozalAut19}), while it is immediate to see that Assumption~\ref{ass:detectable} is necessary for solving Problem~\ref{problem}.

\subsection{Hybrid observer structure}
    \begin{subequations} \label{eq:observer}
    Based on the multi-hop decomposition, each observer located at node $i \in \V$ is defined by the impulsive system 
    \begin{align}
    \label{eq:obs_flow}
      & \hspace{-5pt}  \dot {\hat x}_i = A \hat x_i + W_{i,0} L_i (y_i - C \hat x_i), \quad t \notin \{k T, k \in {\mathbb N} \}, \\
    \label{eq:obs_jump} \begin{split}
          &\hspace{-5pt} {\hat x}_i^+ = \hat x_i + \sum\nolimits\limits_{\varrho=1}^{\ell_i} \sum\nolimits\limits_{j \in \neigh} W_{i,\varrho} N_{i,j,\varrho} W_{j,\varrho-1}^\top \left( y_{i,j} - \hat x_i \right), \\ 
          &\hspace{135pt}t \in \{k T, k \in \mathbb N\},
      \end{split}
    \end{align}
    \end{subequations}
    where $W_{i,\varrho}, i \in \V, \varrho \in \{0, \ldots, \ell_i\}$ are defined in \eqref{eq:W_i_rho}.
The architecture of observer \eqref{eq:observer} comprises three main parts:
\begin{enumerate}
        \item \textit{a continuous-time copy of the observed plant}, 
        \item \textit{a local output injection}, which can be performed in continuous time as well, since each sensor has access to its own measurement in continuous time,
        \item \textit{a consensus correction term}, based on the information coming from the neighbors of agent $i$. Due to the communication scheme,
        this output injection corresponds to an impulsive modification of $\hat x_i$ in the subspace generated by matrices $W_{i,\varrho}$,  $\varrho \in \{ 1,\dots,\ell_i\}$. 
    \end{enumerate}
In \eqref{eq:observer}, the information coming from the neighbors is not available at all times, which makes the impulsive architecture relevant in practical implementation.

   Observer \eqref{eq:observer} requires selecting matrices $L_i$ and $N_{i,j,\varrho}$, of appropriate dimension, for all $(i,j,\varrho)$ in $\mathcal V\times \mathcal N_i \times \{0,\dots,\ell_i\}$. Their selection is addressed in Section~\ref{sec:feasibility}.

  \subsection{Stability guarantees}
    The following property presents sufficient conditions for the observer in \eqref{eq:observer} to solve Problem~\ref{problem}.
    We will show in Section~\ref{sec:feasibility} that, under the necessary Assumption~\ref{ass:detectable}, there always exist observer gains satisfying this property.
    \begin{property}
      \label{prop:gains}
      For each agent $i$, the local observer gains $L_i$ are chosen so that
      \begin{align}
      \label{eq:gain_CT}
        \overline A_{i,0} := \left( W_{i,0}^\top A - L_i C_i \right) W_{i,0}
      \end{align}
      is Hurwitz, with spectral abscissa $\overline \alpha \leq -\alpha$.  The consensus gains $N_{i,j,\varrho}$,  for all $\varrho \in \left\{ 1, \ldots, \ell_i \right\}$, are chosen so that
      \begin{align}
      \label{eq:gain_DT}
        \hspace{-5pt} \overline A_{i,\varrho} := \mathrm{e}^{(W_{i,\varrho}^\top A W_{i,\varrho}) T} \Big( I - \Big( \sum\nolimits\limits_{j \in \neigh} N_{i,j,\varrho} W_{j,\varrho-1}^\top \Big) W_{i,\varrho} \Big)
      \end{align}
      is Schur, with spectral radius $\overline \beta \in [0, \mathrm{e}^{-\alpha T}]$.
    \end{property}
    We are now in position to state the main result, which will be proven in Section~\ref{sec:main_proof}.
    \begin{theorem}
    \label{thm:iss}
      Given $\alpha > 0$, if the gains of observer \eqref{eq:observer} satisfy Property~\ref{prop:gains}, then this observer solves Problem~\ref{problem}.
    \end{theorem}

  \subsection{Feasibility guarantees and design guidelines}
  \label{sec:feasibility}
    Following \cite[Th.14]{delNozalAut19}, we show next that Assumption~\ref{ass:detectable} is not only necessary, but also sufficient for solving Problem~\ref{problem}.
    \begin{theorem}
    \label{thm:feasibility}
      Under Assumption~\ref{ass:detectable}, there always exist gains $L_i$ and $N_{i,j,\varrho}$ satisfying Property~\ref{prop:gains}.
    \end{theorem}
    \begin{proof}
      The existence of matrices $L_i$ is trivial, as \eqref{eq:gain_CT} considers only the locally observable modes.
      Concerning the matrices $N_{i,j,\varrho}$, we can rewrite each matrix $\overline{A}_{i,\varrho}$ in \eqref{eq:gain_DT} as
      \begin{align*}
        \overline{A}_{i,\varrho} &= E_{i,\varrho} - \overline N_{i,\varrho} \overline C_{i,\varrho},
      \end{align*}
      where $E_{i,\varrho} := \mathrm{e}^{W_{i,\varrho}^\top A W_{i,\varrho} T}$, {$\overline N_{i,\varrho} := E_{i,\varrho} \col\nolimits\limits_{j \in \neigh} \left( N_{i,j,\varrho}^\top \right)^\top$} and $C_{i,\varrho} := \col\nolimits\limits_{j \in \neigh} \left( W_{j,\varrho-1}^\top \right) W_{i,\varrho}$. Using this structure, we know from linear systems theory that, if the pair $\left( \overline C_{i,\varrho}, E_{i,\varrho} \right)$ is observable, then it is possible to find $\overline N_{i,\varrho}$ that places the eigenvalues of $\overline{A}_{i,\varrho}$ anywhere in the unit circle.
      Note also that any selection of $\overline N_{i,\varrho}$ corresponds to a unique selection of the gains $N_{i,j,\varrho}$ for all $j \in \neigh$, because $E_{i,\varrho}$ is non singular for any $T>0$.
      We then complete the proof by showing observability of $\left( \overline C_{i,\varrho}, E_{i, \varrho} \right)$.

      According to the PBH test for observability (see \cite[Th.15.9]{HespanhaBook}), the pair $\left( \overline C_{i,\varrho}, E_{i,\varrho} \right)$ is observable if and only if $
        \mathrm{rank} \left( \left[\begin{smallmatrix}
          E_{i,\varrho} - \lambda I \\ \overline C_{i,\varrho} 
        \end{smallmatrix}\right] \right) = n_{i,\varrho}, \; \forall \lambda \in \mathbb C,$
      where the size $n_{i,\varrho}$ of $E_{i,\varrho}$ corresponds to the number of columns of $W_{i,\varrho}$.
      In turn, this test is true if $\overline C_{i,\varrho}$ is full column rank.
      From item (iii) of \cite[Lemma 3]{delNozalAut19} we know that the image of $W_{i,\varrho}$ is a subset of the sum of the images of $W_{j,\varrho-1}, j \in \neigh$, or equivalently
      \begin{align}
      \label{eq:WW}
        W_{i,\varrho} = \col\nolimits\limits_{j \in \neigh} \left( W_{j,\varrho-1}^\top \right)^\top Q,
      \end{align}
      where $Q$ is some selection matrix.
      Since $W_{i,\varrho}$ is full column rank by definition, we may use \eqref{eq:WW} to obtain
      $Q^\top \overline C_{i,\varrho} 
        = Q^\top \col\nolimits\limits_{j \in \neigh} \left( W_{j,\varrho-1}^\top \right) W_{i,\varrho} = W_{i,\varrho}^\top W_{i,\varrho}$,
      which implies
      $  \mathrm{rank} \left( \overline C_{i,\varrho} \right) \geq \mathrm{rank} \left( Q^\top \overline C_{i,\varrho} \right) = n_{i,\varrho}$,
      and since $\overline C_{i,\varrho}$ has exactly $n_{i,\varrho}$ columns, the equality is immediately obtained, thus concluding the proof.
    \end{proof}
    
    \begin{remark}
      A good practical rule is to tune the observer gains so that $\overline \alpha \ll -\alpha$ while $\overline \beta$ is closer to $\mathrm{e}^{-\alpha T}$. Then the local observation errors converge to zero faster than the consensus errors, thus improving the transient behaviour.
    \end{remark}

  \section{Error dynamics}

    Define the estimation error of each agent as
    \begin{align}
      \label{eq:error}
      e_i := x - \hat x_i, \quad i \in \V,
    \end{align}
    and apply the multi-hop decomposition to the error dynamics, which yields, for each $i \in \V$ and each $\varrho \in \{0, \ldots, \ell_i + 1\}$
    \begin{align}
      \label{eq:error_multihop}
      \varepsilon_{i,\varrho} := W_{i,\varrho}^\top e_i \quad\longrightarrow\quad e_i := \sum\nolimits\limits_{r = 0}^{\ell_i+1} W_{i,r} \varepsilon_{i,r}.
    \end{align}
    
    Define also the index sets
    \begin{align}
      \label{eq:V_rho}
      \V_\varrho := \left\{ i \in \V : \ell_i + 1 \geq \varrho \right\}, \quad \varrho \in \{0, \ldots, \overline \ell \},
    \end{align}
    where $\overline \ell := \max\nolimits\limits_i \{ \ell_i + 1 \}$.
    Then, for each $\varrho \in \{0, \ldots \overline \ell \}$, we can denote with $\varepsilon_{\varrho}$ the vector obtained by stacking all the transformed error coordinates at hop $\varrho$, and with $\varepsilon$ the vector obtained stacking all the $\varepsilon_\varrho$, namely
    \begin{align}
    \label{eq:epsilon_rho}
            \varepsilon_{\varrho} := \col\nolimits\limits_{i \in \V_\varrho} \left( \varepsilon_{i,\varrho} \right), \quad\quad
            \varepsilon := \col\nolimits\limits_{\varrho \in \{ 0, \ldots, \overline \ell \}} \left( \varepsilon_\varrho \right) \in \real^{n_\varepsilon}.
    \end{align}

    Given these definitions, the error dynamics can be described by means of a hybrid dynamical system.
    To this end, we augment the error state $\varepsilon$ with a timer element $\tau$, which will be used to trigger periodic jumps, representing the sampled-data exchange of information.
    In particular, we prove in the sequel that the error dynamics results in
    \begin{align}
      \label{eq:error_hs}
      \hspace{-10pt} \begin{array}{ll}
        \begin{bmatrix}
          \dot \varepsilon \\ \dot \tau
        \end{bmatrix} = \begin{bmatrix}
          A_\varepsilon \varepsilon + R d \\ 1
        \end{bmatrix}, &(\varepsilon, \tau) \in \C, \\
        \begin{bmatrix}
          \varepsilon^+ \\ \tau^+
        \end{bmatrix} = \begin{bmatrix}
          J_\varepsilon \varepsilon + S w \\ 0
        \end{bmatrix}, &(\varepsilon, \tau) \in \D, \\
      \end{array}
    \end{align}
    where the state evolves in $X := \real^{n_\varepsilon} \times \left[ 0, T \right]$ and the flow and jump sets are defined as
    \begin{gather}
      \label{eq:error_sets}
      \C := X, \quad\quad \D := \left\{ (\varepsilon, \tau) \in X : \tau = T \right\}.
    \end{gather}
    The matrices in \eqref{eq:error_hs} are defined as
    \begin{gather}
    \label{eq:upper_triang_matrices}
      A_\varepsilon := \smallmat{
        A_{\overline \ell} &&&& \\
        & A_{\overline \ell - 1}&&\star& \\
        &&\ddots&& \\
        &0&&A_1& \\
        &&&&A_0
      }, \quad J_\varepsilon := \smallmat{
        \Delta_{\overline \ell} &&&& \\
        & \Delta_{\overline \ell - 1}&&\star& \\
        &&\ddots&& \\
        &0&&\Delta_1& \\
        &&&&\Delta_0
      }, \\ \label{eq:noise_matrices}
      R := \col\nolimits\limits_{\varrho \in \{0, \ldots, \overline \ell\}} \left( R_\varrho \right), \quad
      S := \col\nolimits\limits_{\varrho \in \{0, \ldots, \overline \ell\}} \left( S_\varrho \right),
    \end{gather}
    where $\star$ represents some possibly non-zero terms.
    Using the definition of $\V_\varrho$ in \eqref{eq:V_rho} and adopting the convention that the unobservable modes of node $i$ correspond to $N_{i,j,\ell_i+1} := 0$ for all $j \in \neigh$, the entries of \eqref{eq:upper_triang_matrices}-\eqref{eq:noise_matrices} are given by
    \begin{subequations}
    \label{eq:matrix_entries}
      \begin{align}
            &\hspace{-7pt} A_\varrho := \begin{cases}
              \diag\nolimits\limits_{i \in \V} \left(W_{i,0}^\top A W_{i,0} - L_i C_i W_{i,0} \right), &\mbox{if } \varrho = 0, \\
              \diag\nolimits\limits_{i \in \V_\varrho} \left(W_{i,\varrho}^\top A W_{i,\varrho} \right), &\hspace{-35pt} \mbox{if } \varrho \in \{ 1, \ldots, \overline \ell \},
            \end{cases} \\ 
            &\hspace{-7pt} R_\varrho := \begin{cases}
              \begin{bmatrix}
                \diag\nolimits\limits_{i \in \V} \left( -L_i \right) \; \big| \; \col\nolimits\limits_{i \in \V} \left( W_{i,0}^\top \right) 
              \end{bmatrix}, &\mbox{if } \varrho = 0, \\
              \begin{bmatrix}
                \; \boldsymbol{0} \; \big| \; \col\nolimits\limits_{i \in \V_\varrho} \left( W_{i,\varrho}^\top \right)
              \end{bmatrix}, &\hspace{-30pt} \mbox{if } \varrho \in \{ 1, \ldots, \overline \ell \},
            \end{cases} \\
            &\hspace{-7pt} \Delta_\varrho := \begin{cases}
              I, \hspace{44mm} \mbox{if } \varrho = 0, \\
              \diag\nolimits\limits_{i \in \V_\varrho} \left( I - \left( \sum\nolimits\limits_{j \in \neigh} N_{i,j,\varrho} W_{j,\varrho-1}^\top \right) W_{i,\varrho} \right), \\
              \hspace{35mm} \mbox{if } \varrho \in \{ 1, \ldots, \overline \ell \}.
            \end{cases} 
          \end{align}
           \begin{align}
            &\hspace{-7pt} S_\varrho := \begin{cases}
              \boldsymbol{0}, \hspace{120pt} \mbox{if } \varrho = 0, \\
              \col\nolimits\limits_{i \in \V_\varrho} \left( S_{i,\varrho} \right), \hspace{50pt} \mbox{if } \varrho \in \{ 1, \ldots, \overline \ell \},
            \end{cases} \\
            &\hspace{-7pt} S_{i,\varrho} = -\begin{bmatrix}
              \boldsymbol{0} \quad \col\nolimits\limits_{j \in \neigh} \left( W_{j,\varrho-1} N_{i,j,\varrho}^\top \right)^\top \quad \boldsymbol{0}
            \end{bmatrix},
          \end{align}
    \end{subequations}
    where the only non-zero element of the block vector $S_{i, \varrho}$ is in position $i$.
    We can finally state the following lemma, whose proof can be found in the appendix.
    \begin{lemma}
      \label{lem:error_dyn}
      Consider the transformed error coordinates in \eqref{eq:error_multihop}, the vector $\varepsilon$ defined in \eqref{eq:epsilon_rho} and the noise vectors $d$ and $w$ in \eqref{eq:noise}.
      Then, $\varepsilon$ evolves according to \eqref{eq:error_hs}-\eqref{eq:matrix_entries}.
    \end{lemma}
    Note that 
    the flow and jump sets in \eqref{eq:error_sets} overlap on their boundaries.
    This is typical in well-posed hybrid systems \cite[Ch. 6]{TeelBook} where
$\mathcal C$ and $\mathcal D$ are closed. Well-posedness ensures robustness of stability \cite[Ch. 7]{TeelBook}, and even though having overlapping $\mathcal C$ and $\mathcal D$ might generate non-unique solutions, this is not a concern in the hybrid framework where Lyapunov-based proofs apply to all solutions. Moreover, uniqueness of solutions can be proved for the system \eqref{eq:error_hs}-\eqref{eq:matrix_entries} and a similar situation occurs in the famous bouncing ball example \cite[Ex 1.1]{TeelBook}.

  \section{Numerical simulations}
    Consider the continuous-time system \eqref{eq:plant} with
    \begin{align*}
      A = \left[\begin{smallmatrix}
        0&1&0&0 \\ -1&0&0&0 \\ 0&0&0&2 \\ 0&0&-2&0
      \end{smallmatrix}\right], \quad\quad
      C_i = e_i^\top, \quad i \in \left\{1, \ldots, 4 \right\},
    \end{align*}
    where $e_i$ represents the $i$-th vector of the Euclidean basis, i.e. a vector with all zeros and a one in position $i$.
    The agents are interconnected with a directed ring graph, i.e. agent 1 sends information to agent 2, 2 to 3, and so on.
    
    \begin{figure}[hbt]
        \centering
        \includegraphics[width = 0.8\columnwidth]{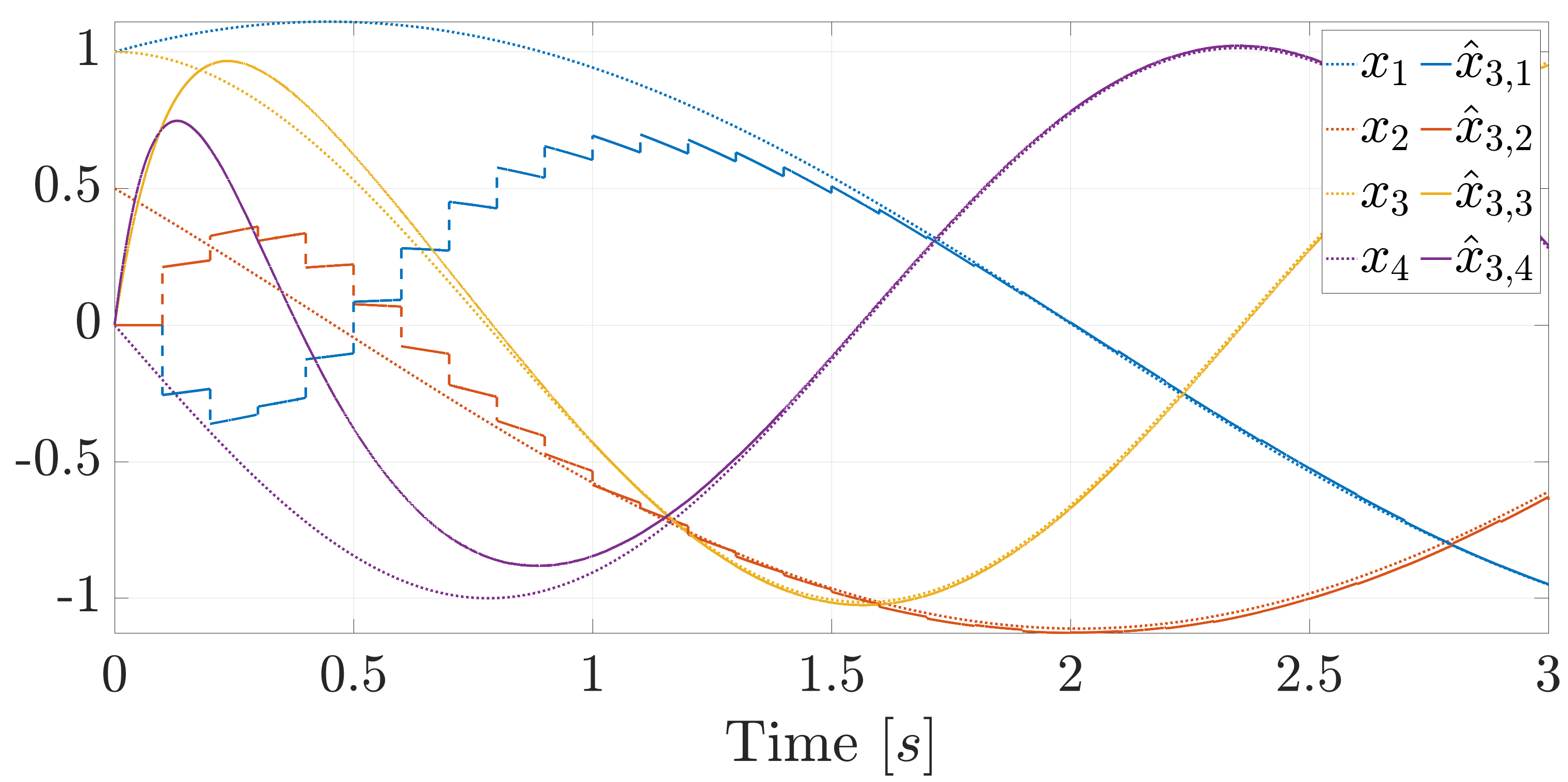}
        \caption{Evolution of the estimates of agent 3, compared with the plant states.}
        \label{fig:plots}
    \end{figure}

    \begin{figure}[hbt]
        \centering
        \includegraphics[width = 0.8\columnwidth]{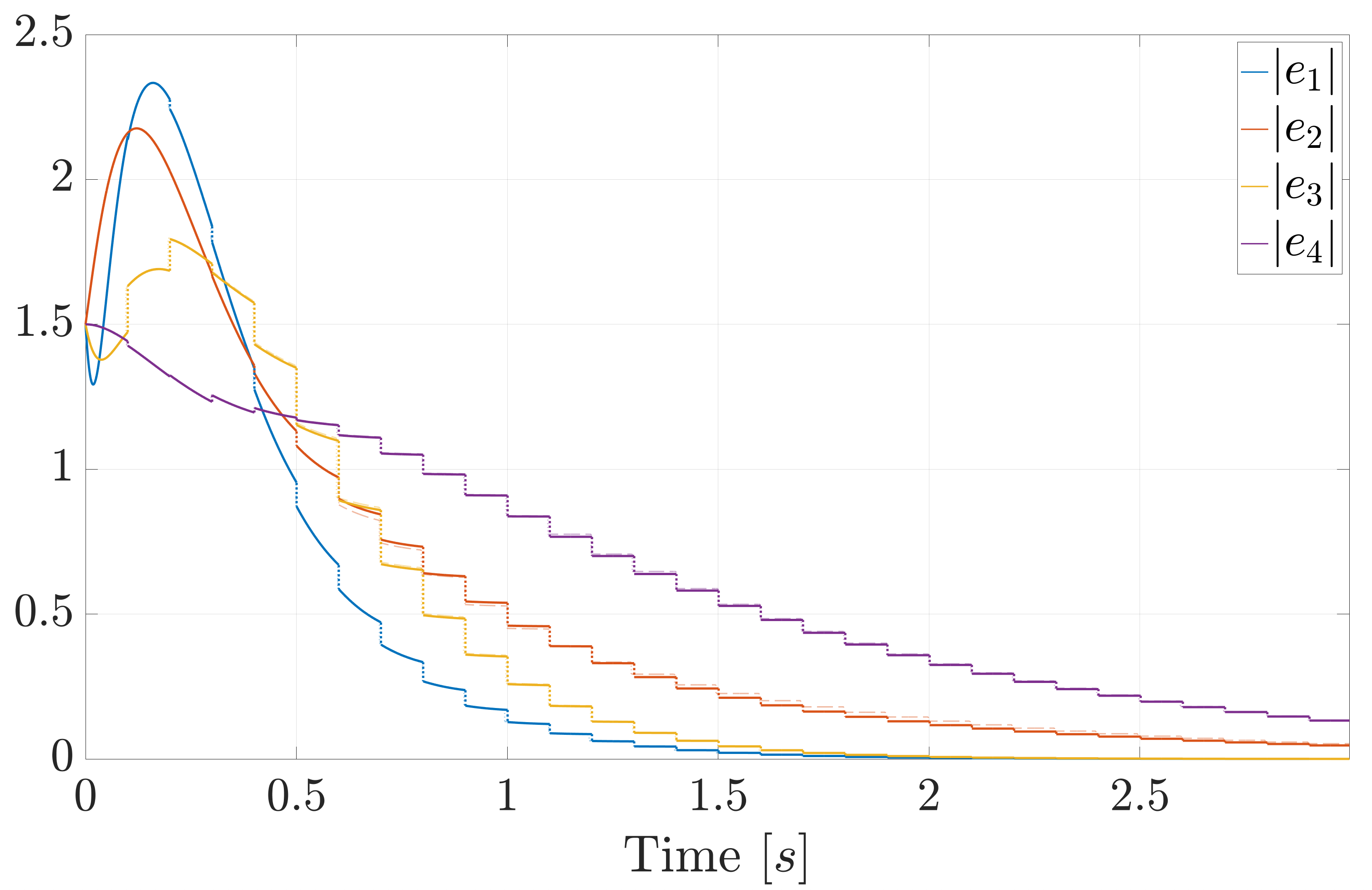}
        \caption{Evolution of the estimation error norms: comparison between the nominal case (opaque line) and the system in presence of timer jittering (transparent line).}
        \label{fig:jittering}
    \end{figure}

    \begin{figure}[hbt]
        \centering
        \includegraphics[width = 0.8\columnwidth]{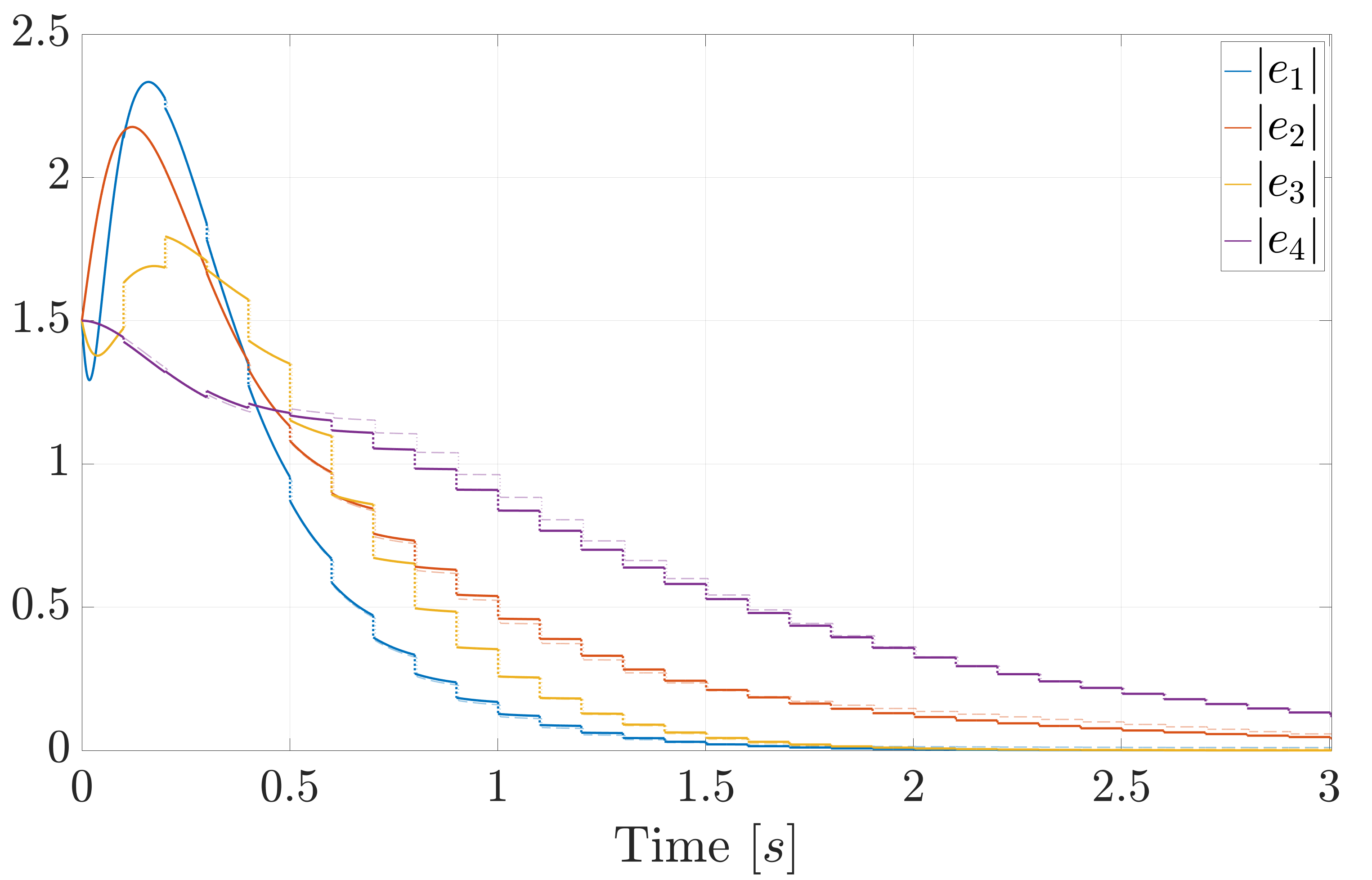}
        \caption{Evolution of the estimation error norms: comparison between the nominal case (opaque line) and the system in presence of transmission delays (transparent line).}
        \label{fig:delay}
    \end{figure}

    \begin{figure}[hbt]
        \centering
        \includegraphics[width = 0.8\columnwidth]{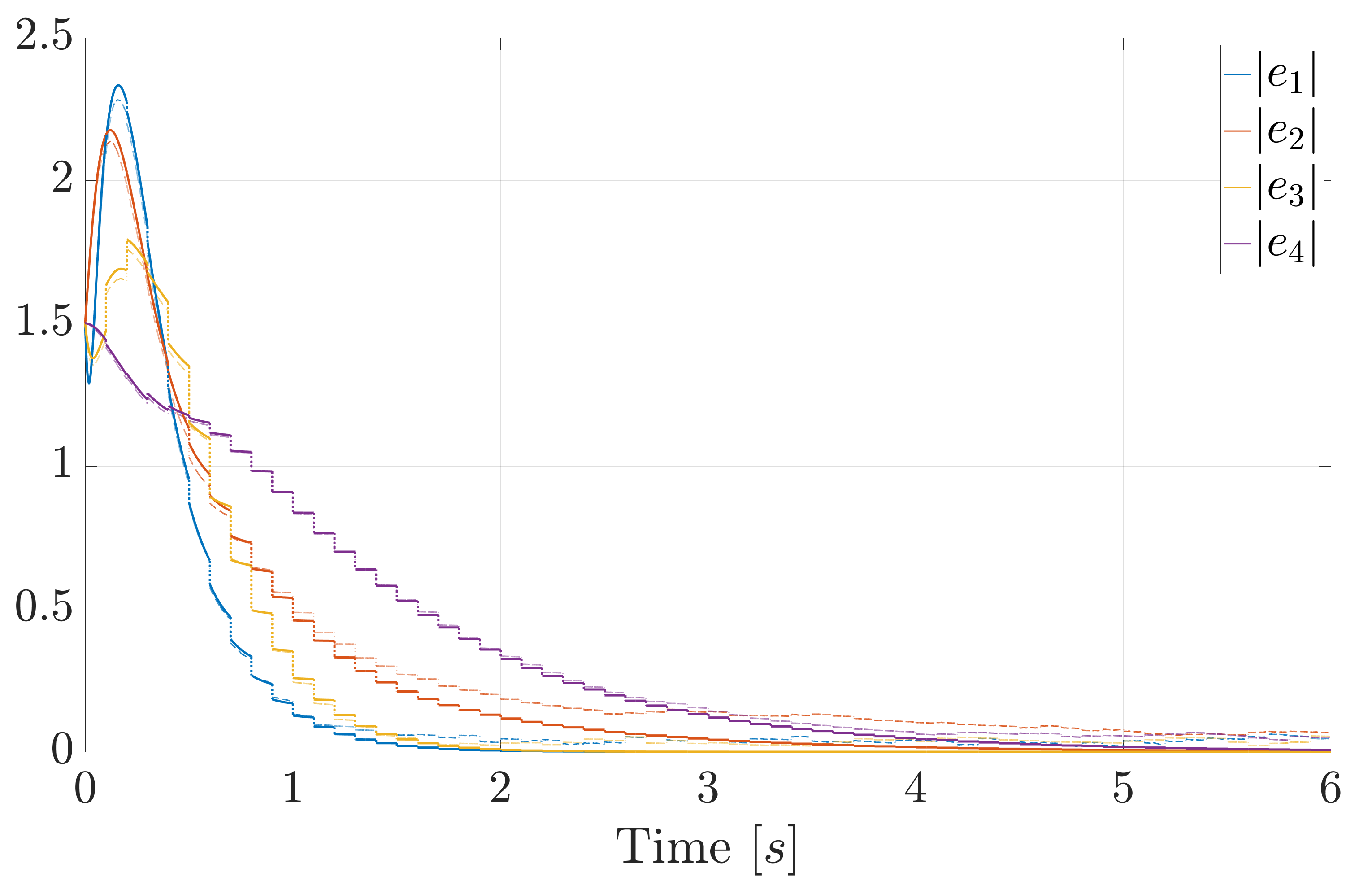}
        \caption{Evolution of the estimation error norms: comparison between the nominal case (opaque line) and the system in presence of generic disturbances $d$ and $w$, as discussed in this paper (transparent line).}
        \label{fig:noise}
    \end{figure}

    We report the simulation results for the proposed observer hereafter.
    For this simulation, we selected $T = 0.1s$ and $\alpha = 1$.
    The observer gains were chosen so that Property~\ref{prop:gains} was satisfied with $\overline \alpha = -5 \alpha$ and $\overline \beta = \mathrm{e}^{-\alpha T}$.
    The disturbances $d$ and $w$ were simulated as random signals, with $\infnorm{d} = 0.04$ and $\infnorm{w} = 0.02$. In Figure~\ref{fig:plots}, we can see the comparison among the state estimates of agent 3 and the plant states.
    We can see that the estimates of the locally observable modes, namely $\hat x_{3,3}$ and $\hat x_{3,4}$, converge continuously to the corresponding plant states.
    On the other hand, convergence of the locally unavailable states $\hat x_{3,1}$ and $\hat x_{3,2}$ is achieved through the jumps emerging from the periodic information exchange with the neighbouring agents.

    In Figures~\ref{fig:jittering}-\ref{fig:noise}, the evolution of the error estimates in the nominal case is compared with three different non-ideal cases.
    In Figure~\ref{fig:jittering}, we show the result of a simulation with timer jittering, i.e. each agent is governed by its own timer state, and the information is transmitted when the value of the timer is in the interval $[T - \epsilon_\tau, T + \epsilon_\tau]$. In this simulation, we selected $\epsilon_\tau = T/10 = 10^{-2}$s.
    Figure~\ref{fig:delay} shows the estimation errors in the presence of communication delays, in particular the update of each agent happens with a fixed delay of $\delta = 5 \cdot 10^{-3}$s.
    Lastly, in Figure~\ref{fig:noise} we compare the nominal case with the system affected by nonzero norm-bounded Gaussian selections of $d$ and $w$, as per our main assumption; the selected values of the infinity norms are the same as the ones used for the simulation that produced Figure~\ref{fig:plots}.

\section{Proof of Theorem~\ref{thm:iss}}
  \label{sec:proof}
  \subsection{Nonsmooth Lyapunov function}
    Consider the Lyapunov function candidate
    \begin{align}
      \label{eq:lyap_fun}
      \begin{split}
        V(\varepsilon, \tau) &\!:= 
        \!\sqrt{\!\varepsilon^{\!\top} \overline P(\tau) \varepsilon} 
        := 
        \sqrt{\varepsilon^\top \mathrm{e}^{A_\varepsilon^\top (T - \tau)} P \mathrm{e}^{A_\varepsilon (T - \tau)} \varepsilon},
      \end{split}
    \end{align}
    with  $P\! >\! 0$  and  $A_\varepsilon$ defined in \eqref{eq:upper_triang_matrices} and introduce the attractor
    \begin{align}
      \mathcal A := \left\{ (\varepsilon, \tau) \in X : \varepsilon = 0 \right\},
    \end{align}
Notice that $V$ is smooth outside $\mathcal A$, so we can define its gradient $\nabla V (\varepsilon, \tau)$ for all $(\varepsilon, \tau) \in X \setminus \mathcal A$.

  \subsection{Properties of $V$ and some useful results}
    We will now present some properties of triangular and exponential matrices, needed to prove the main stability result.
    Their proofs are trivial applications of linear algebra concepts, thus they are not reported.
    \begin{lemma}
    \label{lem:properties}
      Consider $A,B \in \real^{n \times n}$.
      The following holds
      \begin{enumerate}[(i)]
        \item if $A, B$ are upper triangular, then $AB$ is also upper triangular, and $(ab)_{ii} = a_{ii} b_{ii}$; \label{lem:upper_triangular_product}
        \item if $A$ is upper triangular and $B = \mathrm{e}^{A t}$, then $B$ is also upper triangular, and $b_{ii} = \mathrm{e}^{a_{ii} t}$ \label{lem:upper_triangular_exp}
        \item if $\lambda$ is an eigenvalue of $A$, $\mathrm{e}^{\lambda t}$ is an eigenvalue of $\mathrm{e}^{A t}$. \label{lem:exp_eig}
      \end{enumerate}
    \end{lemma}
    Next, we prove some properties of $V$ which will be used to prove the main theorem in the next section.
    \begin{proposition}
    \label{prop:lyap_properties}
      There exist positive scalars $c_1$, $c_2$, $M$, $c_C$, $c_D$ and $\eta \!\leq\! \mathrm{e}^{-\alpha T}$ such that Lyapunov function \eqref{eq:lyap_fun} verifies
      \begin{enumerate}[(i)]
        \item $c_1 |\varepsilon| \leq V(\varepsilon, \tau) \leq c_2 |\varepsilon|, \quad \mbox{for all } \left( \varepsilon,\tau \right) \in X$, \label{eq:sandwich}
        \item $\left| \nabla_\varepsilon V (\varepsilon, \tau) \right| \leq M, \quad \mbox{for all } \left( \varepsilon,\tau \right) \in X \setminus \mathcal A$, \label{eq:grad_bound}
        \item $\dot V \leq 
        c_C |d|, \quad \mbox{for all } \left( \varepsilon,\tau \right) \in \C \setminus \mathcal A$, \label{eq:lyap_flow}
        \item $V^+ \leq - \eta V + c_D |w|, \quad \mbox{for all } \left( \varepsilon,\tau \right) \in \D \setminus \mathcal A$. \label{eq:lyap_jump}
      \end{enumerate}
    \end{proposition}
    \begin{proof}
      We prove the statements one by one.

      \textit{Proof of \eqref{eq:sandwich}:}
      Since $P$ is positive definite, we have
      \begin{align}
      \label{eq:sandwich_intermediate}
        \sqrt{\lambda_m} \left|\mathrm{e}^{A_\varepsilon (T- \tau)} \varepsilon \right| \leq V(\varepsilon, \tau) \leq \sqrt{\lambda_M} \left|\mathrm{e}^{A_\varepsilon (T- \tau)} \varepsilon \right|,
      \end{align}
      where $\lambda_m$ and $\lambda_M$ are the minimum and maximum eigenvalues of $P$, respectively.
      Recalling that, for any non singular matrix $T \in \real^{m \times m}$ and any vector $y \in \real^m$, we have
      \begin{align*}
        |Ty| = \frac{|T^{-1}||Ty|}{|T^{-1}|} \geq \frac{|T^{-1} T y|}{|T^{-1}|} = \frac{1}{|T^{-1}|}|y|,
      \end{align*}
      from \eqref{eq:sandwich_intermediate} we can obtain
      \begin{align*}
        c_1 |\varepsilon| := 
        \sqrt{\lambda_m} \vartheta_1 |\varepsilon| \leq V(\varepsilon, \tau) \leq \sqrt{\lambda_M} \vartheta_2 |\varepsilon| =: c_2 |\varepsilon|,
      \end{align*}
      where
     $
        \vartheta_1 := \min\nolimits\limits_{\tau \in [0, T]} \frac{1}{\left| \mathrm{e}^{A_\varepsilon (\tau - T)} \right|}$ and $
        \vartheta_2 := \max\nolimits\limits_{\tau \in [0, T]} \left| \mathrm{e}^{A_\varepsilon (T - \tau)} \right|,
      $ 
      are both strictly positive, because of continuity and since the matrix exponential is always invertible.

      \textit{Proof of \eqref{eq:grad_bound}:}
      From the explicit calculation of the derivatives we get, for all $(\varepsilon, \tau) \in X \setminus \mathcal A$,
      \begin{align*}
        \left| \nabla_\varepsilon V (\varepsilon, \tau) \right| = \frac{2 \left| \overline P \varepsilon \right|}{2 \sqrt{\varepsilon^\top \overline P \varepsilon}}
        = \sqrt{\frac{\varepsilon^\top \overline P^2 \varepsilon}{\varepsilon^\top \overline P \varepsilon}} \leq \frac{\lambda_M \vartheta_2}{\sqrt{\lambda_m} \vartheta_1} =: M.
      \end{align*}

      \textit{Proof of \eqref{eq:lyap_flow}:}
      Computing the evolution of $V$ along flows of the solutions to \eqref{eq:error_hs}-\eqref{eq:error_sets} we get
      \begin{align*}
        \dot V &=\left\langle \nabla_\varepsilon V (\varepsilon, \tau), A_\varepsilon \varepsilon + R d \right\rangle +
          \left\langle \nabla_\tau V (\varepsilon, \tau), 1 \right\rangle \\
        &= \frac{2 \varepsilon^\top \overline P A_\varepsilon \varepsilon}{2 \sqrt{\varepsilon^\top \overline P \varepsilon}} + 
          \left\langle {\nabla_\varepsilon V (\varepsilon, \tau)}, R d \right\rangle - 
          \frac{2 \varepsilon^\top \overline P A_\varepsilon \varepsilon}{2 \sqrt{\varepsilon^\top \overline P \varepsilon}} \\
        &\leq 
        M |R| |d| =: c_C |d|, \hspace{50pt} \forall (\varepsilon, \tau) \in \C \setminus \mathcal A.
      \end{align*}

      \textit{Proof of \eqref{eq:lyap_jump}:}
      For the evolution of $V$ across jumps of the solutions to \eqref{eq:error_hs}-\eqref{eq:error_sets}, denoting $g(\varepsilon,\tau) := \left( J_\varepsilon \varepsilon + S w, \; 0 \right)$, we have
      \begin{align*}
        V^+ := V(g(\varepsilon, w)) = V(g(\varepsilon, 0)) + V(g(\varepsilon,w)) - V(g(\varepsilon,0)).
      \end{align*}
      Whenever the line segment $\Gamma := \{ \gamma g(\varepsilon,w) + (1-\gamma) g(\varepsilon,0), \gamma \in [0,1] \}$ does not pass through the origin, we can use the mean value theorem to rewrite $V^+$ as
      \begin{align}
      \label{eq:VplusOK}
        V^+ = V(g(\varepsilon,0)) + \left\langle \nabla V(\xi), g(\varepsilon,w) - g(\varepsilon,0) \right\rangle,
      \end{align}
      for some $\xi \in \Gamma$.
      If instead $0 \in \Gamma$, then we can use a similar argument based on perturbations and converging sequences.
      Using \eqref{eq:VplusOK}, we obtain
      \begin{align}
      \label{eq:intermediate}
        V^+ &= V\left( J_\varepsilon \varepsilon, 0 \right) + 
          \left\langle \nabla_\varepsilon V (\xi), S w \right\rangle + \left\langle \nabla_\tau V (\xi), 0 \right\rangle \nonumber \\
        &= \sqrt{\varepsilon^\top J_\varepsilon^\top \mathrm{e}^{A_\varepsilon^\top T} P e^{A_\varepsilon T} J_\varepsilon \varepsilon} + 
          \left\langle \nabla_\varepsilon V (\xi), S w \right\rangle.
      \end{align}
      Consider the first term in the last line of \eqref{eq:intermediate}.
      Since both $A_\varepsilon$ and $J_\varepsilon$ are block upper triangular, items~\eqref{lem:upper_triangular_product} and \eqref{lem:upper_triangular_exp} of Lemma~\ref{lem:properties} ensure that $\mathrm{e}^{A_\varepsilon T} J_\varepsilon$ also is block upper triangular.
      Moreover, the blocks on the main diagonal of $\mathrm{e}^{A_\varepsilon T} J_\varepsilon$ are given by $\mathrm{e}^{A_\varrho T} \Delta_\varrho$, for all $\varrho \in \{0, \ldots, \overline \ell \}$.
      Since these are in turn block diagonal matrices, we conclude that the eigenvalues of $\mathrm{e}^{A_\varepsilon T} J_\varepsilon$ are given by the union of the eigenvalues of the following matrices, issued from \eqref{eq:matrix_entries},
      \begin{align*}
        \begin{cases}
          \mathrm{e}^{\left(W_{i,0}^\top A W_{i,0} - L_i C_i W_{i,0} \right) T}, \hspace{90pt} \forall i \in \V, \\
          \mathrm{e}^{\left(W_{i,\varrho}^\top A W_{i,\varrho} \right) T} \left( I - \left( \sum\nolimits\limits_{j \in \neigh} N_{i,j,\varrho} W_{j,\varrho-1}^\top \right) W_{i,\varrho} \right), \\
          \hspace{130pt} \forall i \in \V, \varrho \in \{1,\ldots, \ell_i\}, \\
          \mathrm{e}^{\left( W_{i,\ell_i+1}^\top A W_{i,\ell_i+1} \right) T}, \hspace{107pt} \forall i \in \V.
        \end{cases}
      \end{align*}
      These correspond to the eigenvalues of 
      \begin{itemize}
          \item $\mathrm{e}^{\overline A_{i,0} T}$ for all $i \in \V$,
          \item $\overline A_{i,\varrho}$ for all $i \in \V, \varrho \in \{1,\ldots,\ell_i\}$,
          \item $\mathrm{e}^{\left( W_{i,\ell_i+1}^\top A W_{i,\ell_i+1} \right) T}$ for all $i \in \V$.
      \end{itemize}
      In view of Property~\ref{prop:gains}, item~\eqref{lem:exp_eig} of Lemma~\ref{lem:properties} and Assumption~\ref{ass:detectable}, we conclude that $\mathrm{e}^{A_\varepsilon T} J_\varepsilon$ has spectral radius $\eta \leq \max\{ \mathrm{e}^{\overline \alpha T}, \overline \beta, \mathrm{e}^{-\alpha T} \} = \mathrm{e}^{-\alpha T} < 1$, which is equivalent to the existence of $P>0$ such that 
$ J_\varepsilon^\top \mathrm{e}^{A_\varepsilon^\top T} P \mathrm{e}^{A_\varepsilon T} J_\varepsilon < \eta^2 P$ (see for example \cite[Th. 8.4]{HespanhaBook}). Then, \eqref{eq:intermediate} becomes
      \begin{align*}
        V^+ &\!\leq \eta \sqrt{\varepsilon^\top \!P \varepsilon}\! +\! M |S| |w| = \eta V \!+\! M |S| |w| = \eta \; V \!+\! c_D |w|, \hspace{50pt} 
      \end{align*}
for all $(\varepsilon, \tau) \in \D \setminus \mathcal A$. \end{proof}

  \subsection{Proof of Theorem 1}
  \label{sec:main_proof}
    We will use a proof technique similar to the one in \cite[Theorem 1]{NesicAut13}, which needs to be carefully revisited in some of its parts, as the conditions are slightly different.
    For each solution $\xi_\varepsilon$ to \eqref{eq:error_hs}-\eqref{eq:error_sets}, we can use items \eqref{eq:lyap_flow}, \eqref{eq:lyap_jump} of Proposition~\ref{prop:lyap_properties} to integrate the value of $v(t,k) = V(\xi_\varepsilon(t,k))$ as follows
    \begin{align*}
     & v(t, k) \leq v(t_k, k) + c_C (t - t_k) \infnorm{d} \\
      & \hspace{3pt} \leq \left( \eta v \left( t_k, k-1 \right) + c_D |w| \right) + c_C T \infnorm{d} \\
      & \hspace{3pt}\leq \eta \left( v \left( t_{k-1}, k-1 \right) + c_C T \infnorm{d} \right) 
         + c_D \infnorm{w} + c_C T \infnorm{d} \\
      & \hspace{3pt}\leq \eta \left( \eta v \left( t_{k-1}, k-2 \right) + c_D \infnorm{w} \right) \\
        & \hspace{70pt} + c_D \infnorm{w} + \left( 1 + \eta \right) c_C T \infnorm{d},
    \end{align*}
    where $\left\{ t_k \right\}$ is the sequence of jump times of $\xi_\varepsilon$, and we used the fact that $0 \leq (t - t_{k}) \leq T$ for all $(t, k) \in \mathrm{dom} \; \xi_\varepsilon$.
    Iterating this process, we obtain
    \begin{align}
    \label{eq:lyap_bound}
    V (\xi_\varepsilon(t, k))\! \leq \!\eta^k V\!(\xi_\varepsilon(0,0))\!+\!\! \sum\nolimits\limits_{h = 0}^{k-1} \!\eta^h \!c_D \!\infnorm{w} \!+\!\! \sum\nolimits\limits_{h = 0}^{k}\! \eta^h T \!c_C \!\infnorm{d}.
    \end{align}
Recalling that $\eta \leq \mathrm{e}^{-\alpha T}$, we have $\eta^k = \mathrm{e}^{\alpha T} \mathrm{e}^{-\alpha (k+1) T}  \leq \mathrm{e}^{\alpha T} \mathrm{e}^{-\alpha t}$, which, together with item~\eqref{eq:sandwich} of Proposition~\ref{prop:lyap_properties} applied to \eqref{eq:lyap_bound}, yields
    \begin{align*}
      |\varepsilon(t,k)| \leq \frac{c_2}{c_1} \mathrm{e}^{\alpha T} \mathrm{e}^{-\alpha t} |\varepsilon(0,0)| + \frac{T c_C \infnorm{d}}{c_1 \left(1 - \eta \right)}  + \frac{c_D \infnorm{w}}{c_1 \left( 1 - \eta \right)} .
    \end{align*}
    Lastly, we note that
    \begin{align*}
      |\varepsilon|^2 
      &= \sum\nolimits\limits_{i \in \V} \sum\nolimits\limits_{\varrho = 0}^{\ell_i + 1}
      |W_{i,\varrho}^\top e_i |^2
= \sum\nolimits\limits_{i \in \V}  e_i^\top W_{i} \sum\nolimits\limits_{\varrho = 0}^{\ell_i + 1} \big( Q_{i,\varrho} Q_{i,\varrho}^\top \big) W_{i}^\top e_i ,
    \end{align*}
    where each $Q_{i,\varrho}$ selects unique columns of $W_i$, because of the orthogonality of the multi-hop decomposition, therefore it is immediate to verify that $\sum\nolimits\limits_{\varrho = 0}^{\ell_i + 1} Q_{i,\varrho} Q_{i,\varrho}^\top = I$.
    This implies
      $|\varepsilon|^2 = \sum\nolimits\limits_{i \in \V} e_i^\top W_i W_i^\top e_i = | \col\nolimits\limits_{i \in \V} (e_i) |^2$.
    Therefore, we showed that \eqref{eq:ISS} holds with
    \begin{align*}
      &\gamma_C := \frac{T c_C}{c_1 \left(1 - \eta \right)} = \frac{T M |R|}{\sqrt{\lambda_m} \vartheta_1 (1-\eta)} = \frac{\lambda_M \vartheta_2}{\lambda_m \vartheta_1^2} \frac{T |R|}{1-\eta}, \\
      &\gamma_D := \frac{c_D}{c_1 \left( 1 - \eta \right)} = \frac{\lambda_M \vartheta_2}{\lambda_m \vartheta_1^2} \frac{|S|}{1-\eta},\qquad \kappa := \frac{c_2}{c_1} \mathrm{e}^{\alpha T},
    \end{align*}
which completes the proof.
    
\section{Conclusions}
\label{sec:conclusions}

    In this work we exploited a continuous-time formulation of the multi-hop decomposition to construct a novel distributed continuous-discrete observer.
    The hybrid nature of the proposed observer allowed us to analyze a scenario where the plant and the local observers work in continuous time, while the information exchange happens only periodically.
    Disturbances acting on the plant, on the measurements and on the communication, were also considered.
    Under some necessary collective detectability assumptions, we showed that it is possible to tune the observer gains to ensure exponential ISS from the noises to the estimation error, with a prescribed exponential convergence rate.
    Simulations were included to show the effectiveness of the proposed estimation algorithm.
    The extension of the proposed algorithm to aperiodic/asynchronous communications, time-varying connection graphs or nonlinear plant dynamics are interesting directions for future work.
    Another interesting possibility to explore is the online adaptation of the sampling period, in order to trade-off estimation performance and energy consumption.

\section*{Appendix: Proof of Lemma~\ref{lem:error_dyn}}

    We start by writing the flow and jump dynamics of the error coordinates.
    From \eqref{eq:plant} and \eqref{eq:observer}, these are given by
    \begin{align*}
      \dot e_i &= \dot x - \dot {\hat x}_i = A x + d_0 - A \hat x_i - W_{i,0} L_i \left( y_i - \hat y_i \right) \\
      &= \left( A - W_{i,0} L_i C_i \right) e_i + d_0 - W_{i,0} L_i d_{i} \\
      e_i^+ &= x^+ - \hat x_i^+ \\
      &= x - \hat x_i - \sum\nolimits\limits_{\varrho = 1}^{\ell_i} \sum\nolimits\limits_{j \in \neigh} W_{i,\varrho} N_{i,j,\varrho} W_{j,\varrho-1}^\top \left( y_{i,j} - \hat x_i \right).
    \end{align*}
    Then, we can use \eqref{eq:error_multihop} to obtain
    \begin{align*}
      &\dot \varepsilon_{i,\varrho} = W_{i,\varrho}^\top A \sum\nolimits\limits_{r=0}^{\varrho} W_{i,r} \varepsilon_{i,r} - W_{i,\varrho}^\top W_{i,0} L_i C_i \sum\nolimits\limits_{r=0}^{\varrho} W_{i,r} \varepsilon_{i,r}, \\
      &\hspace{40mm} + W_{i,\varrho}^\top \left( d_0 - W_{i,0} L_i d_{i} \right) \\
      &\varepsilon_{i,\varrho}^+ = \varepsilon_{i,\varrho} - W_{i,\varrho}^\top \sum\nolimits\limits_{r=1}^{\ell_i} \sum\nolimits\limits_{j \in \neigh} W_{i,r} N_{i,j,r} \\ & \hspace{120pt} W_{j,r-1}^\top \left( \hat x_j - \hat x_i + w_{i,j} \right).
    \end{align*}
    Now, as highlighted in Remark~\ref{rmk:multihop}, we can apply the results in \cite[Lemma 3]{delNozalAut19} and \cite[Lemma 10]{delNozalAut19}, thus the flow dynamics in the previous equation becomes
    \begin{align*}
        &\dot \varepsilon_{i,0} = \left( W_{i,0}^\top A  - L_i C_i \right) W_{i,0} \varepsilon_{i,0} + W_{i,0}^\top d_0 - L_i d_{i}, \\
        &\dot \varepsilon_{i,\varrho} = W_{i,\varrho}^\top A \sum\nolimits\limits_{r=0}^{\varrho} W_{i,r} \varepsilon_{i,r} + W_{i,\varrho}^\top d_0,
    \end{align*}
    while the jump dynamics gives
    \begin{align*}
      \varepsilon_{i,0}^+ &= \varepsilon_{i,0}, \\
      \varepsilon_{i,\varrho}^+ &= \varepsilon_{i,\varrho} - \sum\nolimits\limits_{j \in \neigh} N_{i,j,\varrho} W_{j,\varrho-1}^\top \left( \hat x_j - \hat x_i + w_{i,j} \right) \\
      &= \varepsilon_{i,\varrho} - \sum\nolimits\limits_{j \in \neigh} N_{i,j,\varrho} W_{j,\varrho-1}^\top \bigg(
        \sum\nolimits\limits_{r = 0}^\varrho W_{i,r} \varepsilon_{i,r} \\
       & \hspace{100pt} - W_{j,\varrho-1} \varepsilon_{j,\varrho-1} + w_{i,j} \bigg) \\
      &= \varepsilon_{i,\varrho} - \sum\nolimits\limits_{r=0}^\varrho \sum\nolimits\limits_{j \in \neigh} N_{i,j,\varrho} W_{j,\varrho-1}^\top W_{i,r} \varepsilon_{i,r} \\
      & \hspace{20pt} + \sum\nolimits\limits_{j \in \neigh} N_{i,j,\varrho} \varepsilon_{j,\varrho-1} - \sum\nolimits\limits_{j \in \neigh} N_{i,j,\varrho} W_{j,\varrho-1}^\top w_{i,j},
    \end{align*}
    which correspond to \eqref{eq:error_hs} when expressed in matrix form.
    
  \bibliographystyle{unsrt}
  \bibliography{refs}

\end{document}